\documentclass[12pt]{amsart}
\usepackage{tikz-cd}
\usepackage{enumitem}
\setlist[description]{leftmargin=\parindent,labelindent=\parindent}
\usetikzlibrary{matrix,arrows,decorations.pathmorphing}
\usepackage{amssymb}
\usepackage{amsmath,amscd}
\usepackage{mathrsfs}
\usepackage{url}
\usepackage{cite}
\usepackage{fullpage}
\usepackage{hyperref}
\usepackage{verbatim}
\usepackage{comment}
\usepackage{fancyvrb}
\usepackage{fvextra}
\usepackage{caption}
\usepackage{ytableau}

\newtheorem{thm}{Theorem}
\newtheorem{prop}[thm]{Proposition}
\newtheorem{lem}[thm]{Lemma}

\newtheorem{conj}[thm]{Conjecture}

\theoremstyle{definition}

\newtheorem{rem}[thm]{Remark}

\pgfarrowsdeclare{bad to}{bad to}
{
  \pgfarrowsleftextend{-2\pgflinewidth}
  \pgfarrowsrightextend{\pgflinewidth}
}
{
  \pgfsetlinewidth{0.8\pgflinewidth}
  \pgfsetdash{}{0pt}
  \pgfsetroundcap
  \pgfsetroundjoin
  \pgfpathmoveto{\pgfpoint{-3\pgflinewidth}{4\pgflinewidth}}
  \pgfpathcurveto
  {\pgfpoint{-2.75\pgflinewidth}{2.5\pgflinewidth}}
  {\pgfpoint{0pt}{0.25\pgflinewidth}}
  {\pgfpoint{0.75\pgflinewidth}{0pt}}
  \pgfpathcurveto
  {\pgfpoint{0pt}{-0.25\pgflinewidth}}
  {\pgfpoint{-2.75\pgflinewidth}{-2.5\pgflinewidth}}
  {\pgfpoint{-3\pgflinewidth}{-4\pgflinewidth}}
  \pgfusepathqstroke
}

\renewcommand{\H}{\mathsf{H}}

\newcommand{\zz}{\mathbb{Z}}

\newcommand{\qq}{\mathbb{Q}}
\newcommand{\R}{\mathsf{R}}

\newcommand{\RH}{\mathsf{RH}}
\newcommand{\Bb}{\overline{\mathcal{B}}}
\newcommand{\Hb}{\overline{\mathcal{H}}}

\newcommand{\M}{\mathcal{M}}

\DeclareMathOperator{\ct}{ct}

\renewcommand{\tilde}{\widetilde}

\DeclareMathOperator{\im}{Im}

\usepackage{wrapfig}

\newcommand{\Mb}{\overline{\M}}

\title{The tautological ring of $\Mb_{g,n}$ is rarely Gorenstein}
\author{Samir Canning}

\begin{document}

\begin{abstract}
We prove that the tautological rings $\R^*(\Mb_{g,n})$ and $\RH^*(\Mb_{g,n})$ are not Gorenstein when $g\geq 2$ and $2g+n\geq 24$, extending results of Petersen and Tommasi in genus $2$. The proof uses the intersection of tautological classes with non-tautological bielliptic cycles. We conjecture the converse: the tautological rings should be Gorenstein when $g=0,1$ or $g\geq 2$ and $2g+n<24$. The conjecture is known for $g=0,1$ by work of Keel and Petersen, and we prove several new cases of this conjecture for $\RH^*(\Mb_{g,n})$ when $g\geq 2$.
\end{abstract}

\maketitle
\section{Introduction}
Let $\Mb_{g,n}$ denote the moduli space of stable curves of genus $g$ with $n$ markings. The tautological ring $\R^*(\Mb_{g,n})$ is the $\qq$-subalgebra of the Chow ring of $\Mb_{g,n}$ generated by the decorated strata classes, and its image in cohomology under the cycle class map is $\RH^*(\Mb_{g,n})$. Decorated strata classes are Chow classes corresponding to stable graphs of genus $g$ with $n$ legs, decorated by $\kappa$ classes on the vertices and $\psi$ classes on the legs and half-edges \cite{GraberPandharipande,Pandharipandecalculus}. Finding a complete set of relations among these generators is a fundamental open problem in the intersection theory of $\Mb_{g,n}$. Pixton \cite{Pixton} conjectured relations among the tautological generators\,---\,the \emph{3-spin relations}\footnote{The $3$-spin relations are also known as the generalized Faber--Zagier relations or Pixton's relations.}\,---\,that were shown to hold in cohomology by Pandharipande--Pixton--Zvonkine \cite{PandharipandePixtonZvonkine} and in Chow by Janda \cite{Janda}. No other relations have been found, despite significant investigation \cite{CladerJanda,Jandarelations}.
\begin{conj}[Pixton]\label{Pixton}
The $3$-spin relations are complete for $\R^*(\Mb_{g,n})$ and $\RH^*(\Mb_{g,n})$.
\end{conj}
 Before Pixton's conjecture, it had been speculated that the tautological ring is Gorenstein \cite{FPspeculation,HainLooijenga,Pan3questions}, extending Faber's conjecture on $\R^*(\M_g)$ \cite{Faber}. We say $\R^*(\Mb_{g,n})$ is Gorenstein if all the intersection pairings
\[
\R^i(\Mb_{g,n})\times \R^{3g-3+n-i}(\Mb_{g,n})\rightarrow \R^{3g-3+n}(\Mb_{g,n})\cong\qq
\]
are perfect, and similarly for $\RH^*(\Mb_{g,n})$. If the Gorenstein property holds, a tautological class is zero if and only if it intersects trivially with every tautological class in complementary codimension, thereby providing an algorithm to determine all of the relations in the tautological ring. The Gorenstein property holds in genus $0$ and $1$ \cite{Keel,Petersengenusone}, but Petersen and Tommasi \cite{PetersenTommasi} showed it fails in genus $2$ when there are sufficiently many marked points. Further investigation by Petersen \cite{Petersen} shows that $\R^*(\Mb_{2,n})$ and $\RH^*(\Mb_{2,n})$ are not Gorenstein when $n\geq 20$. Our main result is that the tautological rings are Gorenstein in only finitely many cases when $g\geq 2$. 
\begin{thm}\label{thm:main}
    Neither $\R^*(\Mb_{g,n})$ nor $\RH^*(\Mb_{g,n})$ is Gorenstein if $g\geq 2$ and $2g+n \geq 24$.
\end{thm}
\noindent In fact, it suffices to prove Theorem \ref{thm:main} for cohomology, as a direct calculation \cite[Corollary 2.5]{PetersenTommasi} shows that if $\RH^*(\Mb_{g,n})$ is not Gorenstein, then neither is $\R^*(\Mb_{g,n})$.

The central difficulty in Conjecture \ref{Pixton} is the study of classes in the Gorenstein kernel, classes that intersect trivially with every tautological class in complementary codimension. Conjecture \ref{Pixton} predicts the existence of nonzero classes in the Gorenstein kernel, and proving that such classes are nonzero thus requires the failure of the Gorenstein property. At the moment, there are no known cases where Conjecture \ref{Pixton} holds, but the Gorenstein property fails. The proof of Theorem \ref{thm:main} uses a new method for detecting nonzero classes in the Gorenstein kernel: intersecting with non-tautological classes. Recently, several ways of studying the cycle theory of $\Mb_{g,n}$ beyond the tautological ring have appeared \cite{SvZ, Lian, CLP-ste}. Lian and Schmitt--van Zelm propose adding Hurwitz cycles parametrizing covers of curves of positive genus to the calculus. Such cycles tend to be non-tautological \cite{GraberPandharipande,vanZelm,Liannontaut,AGNES}, whereas the Hurwitz cycles parametrizing covers of genus $0$ curves are tautological \cite{FaberPandharipande}.

\begin{rem}
    Conjecture \ref{Pixton} and the Gorenstein property have natural analogues for moduli spaces of curves of compact type \cite{Pixton, FPspeculation,Pan3questions}. In \cite{CLSGorenstein}, it is shown that if $g\geq 2$ and $2g+n\geq 12$, then neither $\R^*(\M_{g,n}^{\ct})$ nor $\RH^*(\M_{g,n}^{\ct})$ is Gorenstein. Moreover, the restriction of the $3$-spin relations to $\M_6^{\ct}$ are complete \cite{CLSGorenstein}, providing the first known case where Pixton's conjecture holds but the tautological ring is not Gorenstein. 
\end{rem}

The proof of Theorem \ref{thm:main} is carried out in Sections \ref{sec:reduction}, \ref{sec:admcycles}, and \ref{sec:nonzero}. In Section \ref{sec:reduction}, using direct calculations in the tautological ring (Lemmas \ref{lem:increasemarkings} and \ref{lem:increasegenus}), we reduce to the cases when $2g+n=24$. To prove these cases, we recall in Section \ref{sec:admcycles} the result of Petersen--Tommasi \cite{PetersenTommasi} and Petersen \cite{Petersen} when $g=2$: there exists a nonzero tautological class $\alpha\in \RH^{24}(\Mb_{2,20})$ such that $\alpha\cdot \beta=0$ for all $\beta\in \RH^{22}(\Mb_{2,20})$. Let $\xi:\Mb_{2,20}\rightarrow \Mb_{g,n}$ be an iterated self-gluing map, which exists for $(g,n)=(2+k,20-2k)$, $0\leq k\leq 10$, where $k$ is the number of gluings, i.e. $g\geq 2$ and $2g+n=24$. We show that $\xi_*\alpha$ pairs trivially with every tautological class in complementary codimension. It thus suffices to show that $\xi_*\alpha\neq 0$, which we do in Section \ref{sec:nonzero} by pulling back to $\Mb_{2,20}$ and intersecting with a non-tautological bielliptic cycle. The intersection computation is done using pullback formulas for Hurwitz cycles under $\xi^*$ developed by Schmitt--van Zelm \cite{SvZ}. 

A key feature of the proof is that it does not require any explicit knowledge of $\H^*(\Mb_{g,n})$ when $(g,n)\neq (2,20)$. Petersen and Tommasi's proof \cite{PetersenTommasi} that $\RH^*(\Mb_{2,20})$ is not Gorenstein requires essentially complete knowledge of $\H^*(\Mb_{2,20})$, which is obtained through an argument with the Leray spectral sequence and the study of the cohomology of local systems on $\mathcal{A}_2$, the moduli space of principally polarized abelian surfaces. It seems unlikely that this proof strategy will generalize to higher genus. In particular, when $g\geq 4$, $\M_g$ is not an open substack of $\mathcal{A}_g$, so drawing conclusions about the cohomology of $\M_g$ from the cohomology of $\mathcal{A}_g$ is much more difficult. We hope that the idea of reducing to a few pairs $(g,n)$ where both $g$ and $n$ are relatively small may also be useful for studying Conjecture \ref{Pixton}.

\begin{rem}
    The proof of Theorem \ref{thm:main} provides a geometric reason for the failure of the Gorenstein property: the existence of the non-tautological bielliptic cycle in $\H^{22}(\Mb_{2,20})$. The failure of the Gorenstein property for $\M_{g,n}^{\ct}$ does not require the existence of a non-tautological cycle \cite{CLSGorenstein}. Nevertheless, the failure of the Gorenstein property on $\M_g^{\ct}$ is closely related to the existence of a non-tautological class on $\mathcal{A}_g$, the moduli space of $g$-dimensional principally polarized abelian varieties \cite{COPabelian}.
\end{rem}

We conjecture the converse of Theorem \ref{thm:main} should also hold.
\begin{conj}\label{conj:alwaysgorensteinotherwise}
    The tautological rings $\R^*(\Mb_{g,n})$ and $\RH^*(\Mb_{g,n})$ are Gorenstein if and only if $g=0,1$ or $g\geq 2$ and $2g+n<24$.
\end{conj}

\begin{rem}
    If Conjecture \ref{conj:alwaysgorensteinotherwise} is true, the remaining open cases can in principle be checked computationally, using the Sage package \texttt{admcycles} \cite{admcycles} to compute pairing matrices. Because of the number of decorated strata classes, this computation is not currently feasible.
\end{rem}
In cohomology, some cases of Conjecture \ref{conj:alwaysgorensteinotherwise} follow immediately from results in the literature. In particular, the $g=0,1,2$ cases are known \cite{Keel,Petersengenusone,PetersenTommasi,Petersen}. For the $g\geq 3$ cases, one can use that the cohomology rings are generated by tautological classes for many cases when $g$ and $n$ are small \cite[Theorem 1.4]{CL-CKgP}. It follows by Poincar\'e duality that the cohomological tautological ring is Gorenstein in these cases. Table \ref{ckgptable} summarizes what is already known from \cite{Keel,Petersengenusone,PetersenTommasi,Petersen,CL-CKgP}.
\begin{table}[h!]
\centering
\begin{tabular}{|c|c|c|c|c|c|c|c|c|c|} 
 \hline
 $g$ & $0$ & $1$ & $2$ & $3$ & $4$ & $5$ & $6$ & $7$ \\
 \hline
 $c(g)$ & $\infty$ & $\infty$ & $20$ & $9$ & $7$ & $5$ & $3$ & $1$  \\ 
\hline
$24-2g$ & $24$ & $22$ & $20$ & $18$ & $16$ & $14$ & $12$ & $10$ \\
\hline
\end{tabular}
\vspace{.1in}
\caption{$\RH^*(\Mb_{g,n})$ is Gorenstein for $n < c(g)$. When $g\geq 2$ and $n<24-2g$, Conjecture \ref{conj:alwaysgorensteinotherwise} predicts $\RH^*(\Mb_{g,n})$ is Gorenstein.}
\label{ckgptable}
\end{table}
\begin{rem}
    When $g=0,1$, $\R^*(\Mb_{g,n})=\RH^*(\Mb_{g,n})$. We do not know if this equality holds when $g=2$ and $n<c(g)$, but it is true for $g=2$ and $n<10$ \cite[Theorem 1.4]{CL-CKgP}. For $3\leq g \leq 7$ and $n<c(g)$, $\R^*(\Mb_{g,n})=\RH^*(\Mb_{g,n})$ \cite[Theorem 1.4]{CL-CKgP}. Thus, $\R^*(\Mb_{g,n})$ is Gorenstein in these cases.
\end{rem}
We can prove several more cases by studying just the even cohomology. Given $0\leq g \leq 8$, we define $d(g)$ as in Table \ref{ckgptable2}.
\begin{thm}\label{thm:positivegorenstein}
\begin{table}[h!]
\centering
\begin{tabular}{|c|c|c|c|c|c|c|c|c|c|c|} 
 \hline
 $g$ & $0$ & $1$ & $2$ & $3$ & $4$ & $5$ & $6$ & $7$ & $8$ \\
 \hline
 $d(g)$ & $\infty$ & $\infty$ & $20$ & $12$ & $10$ & $8$ & $6$ & $4$ & $1$  \\ 
\hline
$24-2g$ & $24$ & $22$ & $20$ & $18$ & $16$ & $14$ & $12$ & $10$ & $8$ \\
\hline
\end{tabular}
\vspace{.1in}
\caption{The even cohomology of $\Mb_{g,n}$ is tautological for $n< d(g)$. Hence, $\RH^*(\Mb_{g,n})$ is Gorenstein for $n <d(g)$. When $g\geq 2$ and $n<24-2g$, Conjecture \ref{conj:alwaysgorensteinotherwise} predicts $\RH^*(\Mb_{g,n})$ is Gorenstein.}
\label{ckgptable2}
\end{table}
For $g\leq 8$ and $n< d(g)$, the even cohomology of $\Mb_{g,n}$ is tautological. Therefore, $\RH^*(\Mb_{g,n})$ is Gorenstein for $g\leq 8$ and $n<d(g)$. 
\end{thm}
\noindent We do not know the analogous result for $\R^*(\Mb_{g,n})$. In Section \ref{sec:posgorenstein}, we prove Theorem \ref{thm:positivegorenstein} using the stratification of $\Mb_{g,n}$ by topological type and previously known results about the cohomology (both even and odd) of moduli spaces of curves of low genus \cite{CL-CKgP, CLP11, CLP-ste}. When $g\geq 2$, Conjecture \ref{conj:alwaysgorensteinotherwise} predicts that $\RH^*(\Mb_{g,n})$ is Gorenstein for $110$ pairs $(g,n)$. Theorem \ref{thm:positivegorenstein} shows that $\RH^*(\Mb_{g,n})$ is Gorenstein for $61$ of those pairs.

When $\RH^*(\Mb_{g,n})$ is known to have the Gorenstein property, the even cohomology algebra $\H^{2*}(\Mb_{g,n})$ is tautologically generated. It seems likely that the Gorenstein property is simply an accident, occurring only for small $g$ and $n$ because of the non-existence of non-tautological classes. Together with Theorem \ref{thm:main}, Conjecture \ref{alltaut} below implies Conjecture \ref{conj:alwaysgorensteinotherwise} in cohomology by Poincar\'e duality.
\begin{conj}\label{alltaut}
    If $2g+n<24$, $\H^{2*}(\Mb_{g,n})$ is generated by tautological classes.
\end{conj}

\subsection*{Acknowledgments}
I was inspired to work on this problem after a conversation with Dan Petersen. I am grateful to Emily Clader, Carel Faber, Hannah Larson, Dragos Oprea, Rahul Pandharipande, Dan Petersen, Aaron Pixton, and Johannes Schmitt for many related discussions. In particular, Johannes Schmitt suggested the proof of Lemma \ref{lem:increasegenus}. Thanks to Carel Faber, Hannah Larson, Rahul Pandharipande, and Johannes Schmitt for comments on an earlier draft.  This research was supported by a Hermann-Weyl-instructorship from the Forschungsinstitut für Mathematik at ETH Z\"urich.

\section{Reduction to the case $2g+n=24$}\label{sec:reduction}
We first reduce Theorem \ref{thm:main} to the cases when $g\geq 2$ and $2g+n=24$. %The following result is immediate from Faber and Pandharipande's explicit description of a generator for $\RH^{6g-6+2n}(\Mb_{g,n})$\cite[Section 4.1.1]{FaberPandharipande}.

\begin{lem}\label{socles}
The self-gluing and forgetful pushforward maps in top degree
\[
\RH^{6g-8+2n}(\Mb_{g-1,n+2})\rightarrow \RH^{6g-6+2n}(\Mb_{g,n}) \quad \text{and} \quad \RH^{6g-4+2n}(\Mb_{g,n+1})\rightarrow \RH^{6g-6+2n}(\Mb_{g,n})
\]
are isomorphisms.
\end{lem}
\begin{proof}
We have $\H^{6g-6+2n}(\Mb_{g,n})=\RH^{6g-6+2n}(\Mb_{g,n})\cong \qq$ for any $g$ and $n$ \cite{GraberVakil}. The result follows by Poincar\'e duality and the fact that zero cycles in homology push forward nontrivially for degree reasons.
\end{proof}

\begin{lem}\label{lem:increasemarkings}
    If $\RH^*(\Mb_{g,n})$ is not Gorenstein, then $\RH^*(\Mb_{g,n+1})$ is not Gorenstein.
\end{lem}
\begin{proof}
    Let $\pi:\Mb_{g,n+1}\rightarrow \Mb_{g,n}$ be the map forgetting the last marked point. Then 
    \[
    \pi^*:\RH^*(\Mb_{g,n})\rightarrow \RH^*(\Mb_{g,n+1})
    \]
    is injective because for any $\alpha \in \RH^*(\Mb_{g,n})$,
    \[
    \frac{1}{2g-2+n}\pi_*(\pi^*\alpha\cdot \psi_{n+1})=\alpha.
    \]
    Let $\beta \in \RH^*(\Mb_{g,n})$ be a nonzero class such that
    \[
    \beta\cdot \gamma =0
    \]
    for all $\gamma \in \RH^*(\Mb_{g,n})$ of complementary degree. Then $\pi^*\beta\neq 0$, and for all $\delta\in \RH^*(\Mb_{g,n+1})$ of complementary degree
    \[
    \pi_*(\pi^*\beta\cdot \delta)=\beta\cdot \pi_*\delta = 0.
    \]
    Because $\pi_*$ is injective in the top degree by Lemma \ref{socles}, $\RH^*(\Mb_{g,n+1})$ is not Gorenstein.

\end{proof}
\begin{lem}\label{lem:increasegenus}
    If $\RH^*(\Mb_{g})$ is not Gorenstein, then $\RH^*(\Mb_{g+1})$ is not Gorenstein.
\end{lem}
\begin{proof}
Let $\beta\in \RH^k(\Mb_{g})$ be a nonzero class such that
    \[
    \beta\cdot \gamma = 0
    \] 
    for all $\gamma\in \RH^{6g-6-k}(\Mb_{g})$.
    Set $\alpha=\pi^*\beta$, where $\pi:\Mb_{g,1}\rightarrow \Mb_{g}$ is the forgetful map. Let
    \[
    \varphi:\Mb_{g,1}\times \Mb_{1,1}\rightarrow \Mb_{g+1}
    \]
    be the map attaching an elliptic tail.
    For $\zeta\in \R\H^{6g-2-k}(\Mb_{g+1})$, we have
    \[
    \varphi_*(\alpha\otimes 1)\cdot \zeta=\varphi_*((\alpha\otimes1)\cdot\varphi^*\zeta)\in \RH^{6g}(\Mb_{g+1})\cong \qq.
    \]
    Consider the class $(\alpha\otimes 1)\cdot \varphi^*\zeta$.
    Because $\zeta$ is tautological,
    \[
    \varphi^*\zeta\in (\RH^{6g-2-k}(\Mb_{g,1})\otimes \RH^{0}(\Mb_{1,1}))\oplus (\RH^{6g-4-k}(\Mb_{g,1})\otimes \RH^{2}(\Mb_{1,1}))
    \]
    under the K\"unneth decomposition \cite[Proposition 1]{GraberPandharipande}. The product of $\alpha\otimes 1$ with any class in the first summand is trivial for dimension reasons. The product of $\alpha\otimes 1$ with the second summand is zero by the argument in the proof of Lemma \ref{lem:increasemarkings}.
    
    It suffices to show that $\varphi_*(\alpha\otimes 1)\neq 0$. By the self-intersection formula \cite[Equation 11]{GraberPandharipande},
    \[
    \varphi^*\varphi_*(\alpha\otimes 1)=\alpha\cdot(\delta_{1,\emptyset}-\psi_1)\otimes 1 +\alpha\otimes (-\psi_1).
    \]
    Here, $\delta_{1,\emptyset}$ is the boundary divisor on $\Mb_{g,1}$ given by the gluing map $\Mb_{1,1}\times \Mb_{g-1,2}\rightarrow \Mb_{g,1}$.
We see that
\[
\mathrm{pr}_{1*}(\varphi^*\varphi_*(\alpha\otimes 1))=-\frac{1}{24}\alpha\neq 0,
\]
where $\mathrm{pr}_1:\Mb_{g,1}\times \Mb_{1,1}\rightarrow \Mb_{g,1}$ is projection onto the first factor.
Thus, $\varphi_*(\alpha\otimes 1)\neq 0$.
\end{proof}
Combining Lemmas \ref{lem:increasemarkings} and \ref{lem:increasegenus}, we see it suffices to prove Theorem \ref{thm:main} in the cases when $g\geq 2$ and $2g+n=24$.
\section{Admissible covers and cohomology in genus $2$}\label{sec:admcycles}
We collect background information for the proof of Theorem \ref{thm:main} on moduli spaces of admissible double covers and the cohomology of $\Mb_{2,n}$. 
\subsection{Admissible covers and non-tautological cycles}
We consider the stacks $\Hb_{g,G,\xi}$ of admissible $G$-covers of genus $g$ with monodromy data $\xi$, where $G$ is a finite group, as in \cite[Definitions 3.2 and 3.5]{SvZ}. See \cite{SvZ} for complete details. In this paper, we always take $G=\zz/2\zz$. In this special case, the monodromy data is simply a tuple $\xi=(a_1,\dots,a_b)$ of generators for the stabilizer group of points in the fibers of the source curve over $b$ ordered points of the target curve, controlling the ramification behavior. The group law will be written additively, and so $\xi$ has entries in $\{0,1\}$. An entry $a_i=1$ corresponds to marking a ramification point of the source curve over the $i$th marked point of the target curve. An entry $a_j=0$ corresponds to marking two points of the source pairwise switched by the covering involution over the $j$th marked point of the target.

As a rule, we will mark every point of ramification when considering moduli of admissible double covers. That is, if we are given the genus $h$ of the target curve, the monodromy data $\xi$ will contain $2g-4h+2$ entries $1$, the maximal number allowed by the Riemann--Hurwitz formula. Conversely, $\xi$ and $g$ determine $h$.
%\begin{example}\leavevmode
 %   \begin{enumerate}
  %      \item The space $\Hb_{g,G,(1^{2g+2},0^{n_2})}$ parametrizes double covers of genus $0$ curves with all of the ramification points marked and $2n_2$ extra marked points pairwise switched by the involution.
   %     \item The space $\Hb_{g,G,(1^{2g-2},0^{n_2})}$ parametrizes double covers of genus $1$ curves with all of the ramification points marked and $2n_2$ extra marked points pairwise switched by the involution.
   % \end{enumerate}
%\end{example}

The source map 
\[\phi:\Hb_{g,G,\xi}\rightarrow \Mb_{g,2g-4h+2+2n_2}\]
sends an admissible cover to its source curve together with all of the marked points determined by the monodromy data $\xi$ \cite[Theorem 3.7]{SvZ}. When $\xi=(1^{2g-4h+2},0^{n_2})$, we denote the image of the source map by $\Bb_{g\rightarrow h,2g-4h+2,2n_2}$. One can further compose with the forgetful map $\Mb_{g,2g-4h+2+2n_2}\rightarrow \Mb_{g,n_1+2n_2}$, forgetting the first $2g-4h+2-n_1$ points. The image of $\Bb_{g\rightarrow h,2g-4h+2,2n_2}$ under this map is denoted $\Bb_{g\rightarrow h,n_1,2n_2}$. When $h=1$, these loci play a fundamental role because of the following result of Graber--Pandharipande \cite{GraberPandharipande} (when $g=2$ and $n_2=10$) and van Zelm \cite{vanZelm}.
\begin{thm}[Graber--Pandharipande, van Zelm]\label{biellipticnontaut}
    The class $[\Bb_{g\rightarrow 1,0,2n_2}]\in \H^{22}(\Mb_{g,2n_2})$ is non-tautological if $g\geq 2$ and $g+n_2=12$.
\end{thm}
\begin{rem}\label{newproofvanzelm} The proof of Theorem \ref{thm:main} only uses Theorem \ref{biellipticnontaut} when $g=2$. In fact, the $g>2$ cases of Theorem \ref{biellipticnontaut} will follow from the proof of Theorem \ref{thm:main}, using equation \eqref{biellipticplustaut} below and the fact that tautological classes pull back to tautological classes. The argument is not the same as van Zelm's proof of the $g\geq 3$ part of Theorem \ref{biellipticnontaut}, although it is quite similar.
\end{rem}

On the other hand, classes of hyperelliptic loci are tautological \cite[Proposition 1]{FaberPandharipande}.
\begin{thm}[Faber--Pandharipande]\label{hyperelliptictaut}
The class $[\Bb_{g\rightarrow 0, n_1,2n_2}]$ is tautological for any $g,n_1,n_2$.
\end{thm}

\subsection{Cohomology in genus $2$}
The cohomology of $\Mb_{2,n}$ for $n\leq 20$ is completely understood by Petersen--Tommasi~\cite[Theorems A and B]{PetersenTommasi} and Petersen~\cite[Theorem 3.8]{Petersen}.

\begin{thm}[Petersen--Tommasi, Petersen]\label{PetersenTommasi}
Let $k\in \mathbb{N}$ be even.
    \begin{enumerate}
        \item For $n<20$, $\RH^{k}(\Mb_{2,n})=\H^{k}(\Mb_{2,n})$.
        \item For $k\neq 22$, $\RH^{k}(\Mb_{2,20})=\H^{k}(\Mb_{2,20})$.
        \item The image of the pushforward map $\H^{k-2}(\tilde{\partial\M_{2,20}})\rightarrow \H^{k}(\Mb_{2,20})$ lies in $\RH^k(\Mb_{2,20})$, where $\tilde{\partial\M_{2,20}}$ is the normalization of the boundary.
    \end{enumerate}
\end{thm}
\noindent Theorem \ref{PetersenTommasi}(2) combined with Theorem \ref{biellipticnontaut} implies that $\RH^*(\Mb_{2,20})$ is not Gorenstein.

In contrast to the essentially complete result of Theorem \ref{PetersenTommasi}, little is known about the structure of the cohomology of $\Mb_{g,2n_2}$ beyond Theorem \ref{biellipticnontaut} when $g\geq 3$ and $g+n_2=12$. The proof of Theorem \ref{thm:main} relies on pulling back to $\Mb_{2,20}$ and leveraging the complete description of its cohomology, as follows.

By Poincar\'e duality and Theorem \ref{PetersenTommasi}(2), the intersection pairing
\[
\H^{22}(\Mb_{2,20})\times \RH^{24}(\Mb_{2,20})\rightarrow \qq
\]
is perfect. Hence, we can identify $\RH^{24}(\Mb_{2,20})$ with the dual of $\H^{22}(\Mb_{2,20})$. Pick $\alpha\in \RH^{24}(\Mb_{2,20})$ such that $\alpha \cdot [\Bb_{2\rightarrow 1,0,20}]=1$, but $\alpha\cdot \beta=0$ for every tautological class $\beta\in \RH^{22}(\Mb_{2,20})$.\footnote{The class $\alpha$ can be obtained by completing $[\Bb_{2\rightarrow 1,0,20}]$ together with a basis for the tautological subspace $\RH^{22}(\Mb_{2,20})$ to a basis for $\H^{22}(\Mb_{2,20})$ and choosing $\alpha$ to be the first basis vector in the dual basis.} 

Consider the gluing map
\[
\varphi_g: \Mb_{2,20}\rightarrow \Mb_{g,2n_2},
\]
given by identifying the last $g-2$ pairs of marked points on $\Mb_{2,20}$. Note that $g+n_2=12$.  Let $\gamma\in \RH^{24}(\Mb_{g,2n_2})$. By the projection formula,
\[
\varphi_{g*}\alpha\cdot \gamma = \varphi_{g*}(\alpha\cdot \varphi_{g}^*\gamma)=0 
\]
because $\varphi_g^*\gamma$ is tautological. By iteratively applying Lemma \ref{socles}, $\varphi_{g*}$ is an isomorphism in the top degree. Therefore, if $\varphi_{g*}\alpha\neq 0$, $\RH^*(\Mb_{g,2n_2})$ is not Gorenstein. In fact, it suffices to show that $\varphi_{12*}\alpha\neq 0$, as $\varphi_{g}$ factors $\varphi_{12}$. We will show that $\varphi_{12*}\alpha\neq 0$ by computing the intersection of $\varphi_{12*}\alpha$ with the bielliptic locus in genus $12$. By the projection formula,
\[
\varphi_{12*}\alpha\cdot[\Bb_{12\rightarrow 1,0,0}]=\varphi_{12*}(\alpha\cdot\varphi_{12}^*[\Bb_{12\rightarrow 1,0,0}]).
\]
By Lemma \ref{socles} applied iteratively, it suffices to show that $\alpha\cdot\varphi_{12}^*[\Bb_{12\rightarrow 1,0,0}]\neq 0$. We claim that
\begin{equation}\label{biellipticplustaut}
    \varphi_{12}^*[\Bb_{12\rightarrow 1,0,0}]=c[\Bb_{2\rightarrow 1,0,20}]+T,
\end{equation}
where $c\neq 0$ and $T\in \RH^{22}(\Mb_{2,20})$. Assuming \eqref{biellipticplustaut}, we have
\[
\alpha\cdot\varphi_{12}^*[\Bb_{12\rightarrow 1,0,0}]=\alpha\cdot c[\Bb_{2\rightarrow 1,0,20}] + \alpha\cdot T = c \neq 0.
\]
To finish the proof of Theorem \ref{thm:main}, we thus only need to show that equation \eqref{biellipticplustaut} holds.

\begin{rem}
    The non-tautological part of $\H^{22}(\Mb_{2,20})$ spans a copy of the irreducible $\mathbb{S}_{20}$ representation associated to the partition $(2^{10})$ \cite[Remark 3.10]{Petersen}, which makes equation \eqref{biellipticplustaut} surprisingly simple.
\end{rem}

\section{Pulling back $\Bb_{12\rightarrow 1,0,0}$: proof of equation \eqref{biellipticplustaut}}\label{sec:nonzero}
Schmitt and van Zelm \cite{SvZ} developed a method for pulling back admissible covers cycles to boundary strata $\Mb_{\Gamma}$ of $\Mb_{g,n}$ associated to a stable graph $\Gamma$. We apply their method to prove equation \eqref{biellipticplustaut}.

Let $A$ be the stable graph with one vertex of genus $2$ and $10$ loop edges, with the half-edges labeled by $\{1,\dots,20\}$ and edges given by $(1,2),(3,4),\dots,(19,20)$.  Then $\Mb_{A}=\Mb_{2,20}$, and the gluing morphism
\[
\varphi_{A}:\Mb_{A}\rightarrow \Mb_{12}
\]
is the map $\varphi_{12}$, defined in the previous section.
 Let $B$ denote the graph obtained from $A$ by adding $22$ legs, corresponding to the ramification points of a double cover of a genus $1$ curve by a genus $12$ curve. 
%\[
%B = \begin{tikzpicture}[->,>=bad to,baseline=-3pt,node distance=1.3cm,thick,main node/.style={circle,draw,font=\Large,scale=0.5}]
%\node[main node] (A) {2};
%\node at (0,-.7) (B) {...};
%\node at (.3,.5) (B) {...};
%\node at (-.6,-.7) (1) {1};
%\node at (.6,-.7) (2) {22};
%\node at (-.5,.3) (h1) {\footnotesize $h_1$};
%\node at (-.6,-.2) (h2) {\footnotesize $h_2$};
%\draw [-] (A) to [out=200, in=160,looseness=10] (A);
%\draw [-] (A) to (0,-.5);
%\draw [-] (A) to (.5,-.5);
%\draw [-] (A) to (-.5,-.5);
%\draw [-] (A) to (.2,-.5);
%\draw [-] (A) to (-.2,-.5);
%\draw [-] (A) to [out=80, in=120,looseness=10] (A);
%\draw [-] (A) to [out=160, in=120,looseness=10] (A);
%\draw [-] (A) to [out=340, in=20,looseness=10] (A);
%\draw [densely dotted] (.3,.5) to (.5,.3);
%\end{tikzpicture}
%  \] 
The associated gluing map is denoted by 
 \[
 \varphi_{B}:\Mb_B\rightarrow \Mb_{12,22}.
 \]
 Set $\xi=(1^{22})$. We have the composition
 \[
 \Hb_{12,G,\xi}\xrightarrow{\phi} \Mb_{12,22}\xrightarrow{\pi} \Mb_{12},
 \]
 where the first map is the source curve map and the second map forgets the markings. The image in $\Mb_{12}$ is by definition $\Bb_{12\rightarrow 1,0,0}$. We want to compute the pullback
 \[
 \varphi_{A}^*[\Bb_{12\rightarrow 1, 0 ,0}]=\varphi_A^*\pi_*\phi_*[\Hb_{12,G,\xi}].
 \] 
 \subsection{The fiber product and admissible $G$-graphs}
 The first step is to describe the fiber product of $\varphi_{B}$ and $\phi$, which Schmitt and van Zelm do combinatorially in terms of admissible $G$-graphs \cite[Proposition 4.3]{SvZ}. We summarize what is needed from their work here. We continue to take $G=\zz/2\zz$, as it simplifies the discussion of the monodromy data.

Given an admissible cover in $\Hb_{g,G,\xi}$, we take the dual graph $\Gamma$ of the source curve $C$. The graph $\Gamma$ inherits an action of $G$. In addition to the graph $\Gamma$ with its $G$ action, we record the stabilizer group of each half-edge and leg. This data allows us to remember the ramification behavior at the nodes and markings. The data of $\Gamma$ with the $G$ action and the stabilizer group of each half-edge and leg is called an admissible $G$ graph and will be denoted by $(\Gamma,G)$. See \cite[Section 3.5]{SvZ} for a completely combinatorial definition.
 
 To each admissible $G$ graph $(\Gamma,G)$, there is a space $\Hb_{(\Gamma,G)}$ and a morphism
 \[
 \varphi_{(\Gamma,G)}:\Hb_{(\Gamma,G)}\rightarrow \Hb_{g,G,\xi}
 \]
 whose image is the closure of the locus of admissible covers with admissible $G$ graph $(\Gamma,G)$, defined as follows. Let $V$ be a set of vertices of $\Gamma$ consisting of one vertex from each orbit of the $G$ action. Let $G_v\subset G$ denote the stabilizer of $v\in V$, and let $\xi_{v}$ be the vector of monodromy data associated to the half-edges and legs incident to $v$ (recall that an admissible $G$ graph includes the data of the stabilizer group of each half-edge and leg, allowing for the assignment of monodromy data). Then
 \[
 \Hb_{(\Gamma,G)}=\prod_{v\in V} \Hb_{g(v),G_v,\xi_{v}},
 \]
and the morphism $\varphi_{(\Gamma,G)}:\Hb_{(\Gamma,G)}\rightarrow \Hb_{g,G,\xi}$ is obtained by gluing according to the graph $\Gamma$. That $\varphi_{(\Gamma,G)}$ is well-defined is established in \cite[Page 30]{SvZ}. There is also a natural morphism
\[
\phi_{(\Gamma,G)}:\Hb_{(\Gamma,G)}\rightarrow \Mb_{\Gamma}
\]
obtained by remembering the source curve and forgetting the $G$ action such that the diagram
\[
\begin{tikzcd}
{\Hb_{(\Gamma,G)}} \arrow[r, "{\varphi_{(\Gamma,G)}}"] \arrow[d, "{\phi_{(\Gamma,G)}}"] & {\Hb_{g,G,\xi}} \arrow[d, "\phi"] \\
\Mb_{\Gamma} \arrow[r, "\varphi_{\Gamma}"]                                              & {\Mb_{g,n}}                     
\end{tikzcd}
\]
commutes, where $\varphi_{\Gamma}$ is the gluing map associated to the stable graph $\Gamma$. 

Recall that a morphism of stable graphs $f:\Gamma\rightarrow \Gamma'$ (also called a $\Gamma'$ structure on $\Gamma$) induces a gluing morphism
\[
\varphi_{f}:\Mb_{\Gamma}\rightarrow \Mb_{\Gamma'}.
\]
If $(\Gamma,G)$ is an admissible $G$ graph, then a $\Gamma'$ structure on $(\Gamma,G)$ is simply a morphism of the underlying stable graphs $f:\Gamma\rightarrow \Gamma'$. Such a morphism includes the data of an injection $\beta:E(\Gamma')\hookrightarrow E(\Gamma)$ on the sets of edges. The $\Gamma'$ structure $f:\Gamma\rightarrow \Gamma'$ is called generic if $G\cdot \mathrm{Im}(\beta)=E(\Gamma)$. That is, every edge of $\Gamma$ is in the $G$ orbit of an edge of $\Gamma'$. Given a generic $\Gamma'$ structure $f:\Gamma\rightarrow \Gamma'$, there is a morphism
\[
\phi_{f}:\Hb_{(\Gamma,G)}\xrightarrow{\phi_{(\Gamma,G)}} \Mb_{\Gamma}\xrightarrow{\varphi_f} \Mb_{\Gamma'}.
\]
Let $\mathfrak{H}_{\Gamma'}$ denote the set of isomorphism classes of admissible $G$ graphs $\Gamma$ with generic $\Gamma'$ structures $f:\Gamma\rightarrow \Gamma'$. Setting $\Gamma'=B$, we have the following specialization of \cite[Proposition 4.3]{SvZ}.
\begin{prop}
The diagram
    \begin{equation}\label{cartesian}
        \begin{tikzcd}
{ \coprod_{(\Gamma,G,f)\in \mathfrak{H}_B} \Hb_{(\Gamma,G)}} \arrow[d, "\coprod\phi_f"] \arrow[r, "{\coprod\varphi_{(\Gamma,G)}}"] & {\Hb_{12,G,\xi}} \arrow[d, "\phi"] \\
\Mb_{B} \arrow[r, "\varphi_B"]                                                                             & {\Mb_{12,22}}                         
\end{tikzcd}
    \end{equation}
is Cartesian.
\end{prop}
\subsection{Excess intersection and the pullback formula}
The excess intersection class associated to the diagram \eqref{cartesian} can be computed separately on each component. The excess bundle $E_f$ is by definition the quotient of the normal bundles $E_f=\phi_f^*N_{\varphi_B}/N_{\varphi_{(\Gamma,G)}}$. The top Chern class of $E_f$ is computed in \cite[Proposition 4.7]{SvZ}. Let $\beta:E(B)\rightarrow E(\Gamma)$ be the map on edges induced by the generic $B$ structure $(\Gamma,G,f)$. Let $N$ be a set of representatives for the orbit of the $G$ action on $E(\Gamma)$ such that $N\subset \im(\beta)$. Then 
\[
c_{\mathrm{top}}(E_f)=\prod_{(h,h')\in \im(\beta)\smallsetminus N} (-\psi_h-\psi_h').
\]
By definition, these $\psi$ classes are pulled back from $\Mb_{B}$ under $\phi_f$, as noted in the discussion following \cite[Theorem 4.9]{SvZ}. In fact, since they are associated to edges of $B$, they are pulled back from $\Mb_{A}$ under the composition
\[
\Hb_{(\Gamma,G)}\xrightarrow{\phi_f} \Mb_{B}\xrightarrow{\pi_B} \Mb_{A},
\]
where the map $\pi_{B}:\Mb_{B}\rightarrow \Mb_{A}$ simply forgets legs. Now consider the diagram
\[
\begin{tikzcd}
{ \coprod_{(\Gamma,G,f)\in \mathfrak{H}_B} \Hb_{(\Gamma,G)}} \arrow[d, "\coprod\phi_f"] \arrow[r, "{\coprod\varphi_{(\Gamma,G)}}"] & {\Hb_{12,G,\xi}} \arrow[d, "\phi"] \\
\Mb_{B} \arrow[d, "\pi_B"] \arrow[r, "\varphi_B"]                                                                             & {\Mb_{12,22}} \arrow[d, "\pi"]     \\
\Mb_{A} \arrow[r, "\varphi_{A}"]                                                                                      & \Mb_{12}\,.                          
\end{tikzcd}
\]
Specializing \cite[Theorem 4.12]{SvZ} to the case at hand, we have 
\begin{equation}\label{excesspush} 
\varphi_{A}^*[\Bb_{12\rightarrow 1,0,0}]=\varphi_A^*\pi_*\phi_*[\Hb_{12,G,\xi}]=\sum_{(\Gamma,G,f)\in \mathfrak{H}_B} \pi_{B*}\left(c_{\mathrm{top}}(E_f)\cdot \phi_{f*}[\Hb_{(\Gamma,G)}]\right).
\end{equation}

\subsection{Pullback computation}
We will now establish equation \eqref{biellipticplustaut} using equation \eqref{excesspush}.  We examine each term in the summation on the right hand side of equation \eqref{excesspush}:
    \[
    \sum_{(\Gamma,G,f)\in \mathfrak{H}_B} \pi_{B*}\left(c_{\mathrm{top}}(E_f)\cdot \phi_{f*}[\Hb_{(\Gamma,G)}]\right).
    \]
    We first study a large class of $(\Gamma,G,f)\in \mathfrak{H}_B$ that contribute only tautological classes.

\begin{lem}\label{tautterms}
    Suppose $(\Gamma,G,f)\in \mathfrak{H}_B$ satisfies any of the following properties:

    \begin{enumerate}
        \item $\Hb_{(\Gamma,G)}$ is not of the expected codimension;
        \item $\Gamma$ has no vertices of genus $2$;
        \item $G$ acts nontrivially on the vertices;
        \item The quotient graph $\Gamma/G$ has first Betti number $1$;
        \item $\Gamma$ has loop edges;
        \item $G$ acts trivially on an edge;
        \item $\Gamma$ has vertices $v_1,v_2$ connected by $n\geq 1$ edges and $n\neq 2$.
    
    \end{enumerate}
Then $\pi_{B*}\left(c_{\mathrm{top}}(E_f)\cdot \phi_{f*}[\Hb_{(\Gamma,G)}]\right)$ is tautological.
\end{lem}
\begin{proof}
    Suppose $(\Gamma,G,f)$ is such that (1) holds. Then the excess bundle $E_f$ is of rank at least $1$. Because $c_{\mathrm{top}}(E_f)$ is pulled back from $\Mb_{A}$, the term
    \[
    \pi_{B*}\left(c_{\mathrm{top}}(E_f)\cdot \phi_{f*}[\Hb_{(\Gamma,G)}]\right)
    \]
    is a product of algebraic classes of lower codimension on $\Mb_{A}=\Mb_{2,20}$ by the projection formula. Hence, by Theorem \ref{PetersenTommasi}(2), $\pi_{B*}\left(c_{\mathrm{top}}(E_f)\cdot \phi_{f*}[\Hb_{(\Gamma,G)}]\right)$ is tautological. 
    
    From now on, we suppose $(\Gamma,G,f)$ does not satisfy property (1). Then there is no excess contribution, and we study only the terms of the form $\pi_{B*}\phi_{f*}[\Hb_{(\Gamma,G)}]$. Now suppose $(\Gamma,G,f)$ is such that (2) holds. First, assume that $\Gamma$ has a vertex of genus $1$. Then $\pi_{B*}\phi_{f*}[\Hb_{(\Gamma,G)}]$ is supported on the boundary of $\Mb_{A}$, and hence is tautological by Theorem \ref{PetersenTommasi}(3). We argue similarly if all vertices of $\Gamma$ are of genus $0$.

    From now on, we assume $(\Gamma,G,f)$ does not satisfy property (2), and so has a unique vertex of genus $2$. Suppose that $(\Gamma,G,f)$ is such that $G$ acts non-trivially on the vertices of $\Gamma$. We pick a pair $v_1$ and $v_2$ of genus $0$ vertices interchanged by the $G$ action. The $G$ action fixes all of the legs because the monodromy data is $\xi=(1^{22})$. Therefore, there can be no legs attached to $v_1$ and $v_2$. Under the pushforward $\pi_{B*}$ forgetting the legs, the vertices $v_1$ and $v_2$ are not contracted, meaning $\pi_{B*}\phi_{f*}[\Hb_{(\Gamma,G)}]$ is supported on the boundary of $\Mb_{A}$, and is thus tautological.

    From now on, we assume $(\Gamma,G,f)$ does not satisfy property (3). Suppose that $\Gamma/G$ has first Betti number $1$. It follows that every vertex of the quotient graph is of genus $0$, as the target curve has genus $1$. Therefore, in the product decomposition 
    \[
    \Hb_{(\Gamma,G)}=\prod_{v\in V} \Hb_{g(v),G_v,\xi_v},
    \]
    the term corresponding to the genus $2$ vertex parametrizes admissible hyperelliptic covers. Such spaces contribute only tautological classes by Theorem \ref{hyperelliptictaut}. 
    
    From now on, we assume $(\Gamma,G,f)$ does not satisfy property (4). Suppose that $\Gamma$ has a vertex with a loop edge. The quotient graph $\Gamma/G$ must have a loop edge because $G$ acts trivially on the vertices. This contradicts the assumption that $(\Gamma,G,f)$ does not satisfy property (4). 

    From now on, we assume $(\Gamma,G,f)$ does not satisfy property (5). Consider the case when $\Gamma$ has vertices $v_1$ and $v_2$ connected by an edge $e$ that is fixed under the $G$ action. Because the $B$ structure is generic, $e$ must be an edge of $B$. But $e$ is not a loop edge, whereas all the edges of $B$ are. Therefore, there must exist another path in $\Gamma$ connecting $v_1$ and $v_2$, in which all constituent edges are contracted by the $B$ structure $f:\Gamma\rightarrow B$. But then the quotient graph $\Gamma/G$ must have nonzero first Betti number, as the images of $v_1$ and $v_2$ are connected by at least two edges, contradicting that $(\Gamma,G,f)$ does not satisfy property (4).

    From now on, we assume $(\Gamma,G,f)$ does not satisfy property (6). Suppose that $\Gamma$ has vertices $v_1$ and $v_2$ connected by $n\geq 1$ edges $e_1,\dots,e_n$. If $n$ is odd, then because $G$ acts trivially on the vertices, it must also fix an edge, violating the assumption that it does not satisfy property (6). Therefore, $n$ is even. If $n\geq 4$, then the quotient graph $\Gamma/G$ has first Betti number at least $1$, violating the assumption that it does not satisfy property (4).

\end{proof}

\begin{proof}[Proof of Equation \eqref{biellipticplustaut}]
We calculate using equation \eqref{excesspush}:
\[
\varphi_{A}^*[\Bb_{12\rightarrow 1,0,0}]=\varphi_A^*\pi_*\phi_*[\Hb_{12,G,\xi}]=\sum_{(\Gamma,G,f)\in \mathfrak{H}_B} \pi_{B*}\left(c_{\mathrm{top}}(E_f)\cdot \phi_{f*}[\Hb_{(\Gamma,G)}]\right).
\]
We need only to consider the terms on the right hand side satisfying none of the properties listed in Lemma \ref{tautterms}. By the negation of property (1), such terms are of the form
\[
\pi_{B*}\phi_{f*}[\Hb_{(\Gamma,G)}].
\]
By the negation of properties (2) and (5), there is one genus $2$ vertex, all other vertices are of genus $0$, and there are no loops on any vertex. 
By the negation of property (4), the first Betti number of the quotient graph $\Gamma/G$ is $0$. Hence, the image of the unique vertex of genus $2$ is a vertex of genus $1$. By the negation of property (3), the vertices of genus $0$ map bijectively to vertices of genus $0$ of $\Gamma/G$. By Riemann--Hurwitz, we see that every vertex of $\Gamma$ has two markings. By the negation of properties (6) and (7), any two vertices $v_1$ and $v_2$ of $\Gamma$ are either not connected by edges or are connected by exactly two edges pairwise switched by the $G$ action.

%Note that the graph $\Gamma$ cannot contain an $n$-gon, meaning there is a $n$-tuple of vertices $(v_1,\dots, v_n)$ such $v_i$ and $v_{i+1}$ are connected by two edges switched by the $G$ action for all $1\leq i\leq n-1$, and $v_n$ is connected to $v_0$ by two edges switched by the $G$ action. Then the quotient graph $\Gamma/G$ has first Betti number $1$.

We consider the case where there is a genus $0$ vertex $v_1$ not connected to the genus $2$ vertex. If $v_1$ is connected to at least two other vertices, then the component corresponding to $v_1$ is not contracted under the map forgetting the marked points. It follows that $\pi_{B*}\phi_{f*}[\Hb_{(\Gamma,G)}]$ is supported on the boundary, and is thus tautological. If $v_1$ is connected to only one other vertex, then under the map forgetting the marked points, $v_1$ is contracted. The resulting graph has a genus $0$ vertex with a loop, and thus $\pi_{B*}\phi_{f*}[\Hb_{(\Gamma,G)}]$ is supported on the boundary. 

Finally, we consider the case when no two genus $0$ vertices are connected by edges. Upon forgetting the marked points, all of the genus $0$ components are contracted, and the cycle $\pi_{B*}\phi_{f*}[\Hb_{(\Gamma,G)}]$ is a positive multiple of $[\Bb_{2\rightarrow 1,0,20}]$.

\end{proof}
%Equation \eqref{biellipticplustaut} implies Theorem \ref{thm:main} in cohomology by the argument of Section \ref{sec:admcycles}. To prove it in Chow, we argue as in \cite[Corollary 2.5]{PetersenTommasi}, replacing $2$ with $g$. \qed

\section{Proof of Theorem \ref{thm:positivegorenstein}}\label{sec:posgorenstein}
We prove Theorem \ref{thm:positivegorenstein} by induction on $g$ and $n$. The proof idea is analogous to \cite[Lemma 4.1]{CL-CKgP}, but extra care is needed for possible contributions from odd cohomology in the boundary. The statement holds for $g=0$ \cite{Keel}, $g=1$ \cite{Petersengenusone}, and $g=2$ \cite{Petersen}.
For any $g,n$ and $k$, there is an exact sequence 
\begin{equation}\label{weightexactsequence}
\H^{2k-2}(\tilde{\partial \M_{g,n}})\rightarrow \H^{2k}(\Mb_{g,n})\rightarrow \mathsf{W}_{2k}\H^{2k}(\M_{g,n})\rightarrow 0,
\end{equation}
where $\tilde{\partial \M_{g,n}}$ is the normalization of the boundary. For $3\leq g\leq 8$, $n< d(g)$, and $k$ arbitrary, we have that 
\[
\mathsf{W}_{2k}\H^{2k}(\M_{g,n})=\RH^{2k}(\M_{g,n})
\]
by \cite[Theorem 1.4]{CL-CKgP} and \cite[Theorem 1.10 and Lemma 4.3]{CLP-ste}. Therefore, using the exact sequence \eqref{weightexactsequence}, it suffices to study the boundary contribution.

\begin{table}[h!]
\centering
\begin{tabular}{|c|c|c|c|c|c|c|c|c|c|} 
 \hline
 $g$ & $0$ & $1$ & $2$ & $3$ & $4$ & $5$ & $6$ & $7$ \\
 \hline
 $e(g)$ & $\infty$ & $11$ & $10$ & $9$ & $7$ & $5$ & $3$ & $1$  \\ 
\hline
\end{tabular}
\vspace{.1in}
\caption{The odd cohomology of $\Mb_{g,n}$ vanishes for $n < e(g)$.}
\label{allcohtaut}
\end{table}

Up to quotients by finite groups, the normalization of the boundary $\tilde{\partial \M_{g,n}}$ is the disjoint union of spaces of the form $\Mb_{g-1,n+2}$ and $\Mb_{g_1,n_1+1}\times \Mb_{g_2,n_2+1}$, where $g_1+g_2=g$ and $n_1+n_2=n$. By induction on $g$ and $n$, the even cohomology of $\Mb_{g-1,n+2}$ and $\Mb_{g_1,n_1+1}$ and $\Mb_{g_2,n_2+1}$ are tautological. We thus only need to consider the summands in the K\"unneth decomposition
\[
\H^{2k-2}(\Mb_{g_1,n_1+1}\times \Mb_{g_2,n_2+1})\cong \bigoplus_{p+q=2k-2} \H^{p}(\Mb_{g_1,n_1+1})\otimes \H^{q}(\Mb_{g_2,n_2+1})
\]
when both $p$ and $q$ are odd. It follows from \cite[Theorem 1.4]{CL-CKgP} that for $n_1+1<e(g_1)$ (respectively, $n_2+1<e(g_2)$), as defined in Table \ref{allcohtaut}, all of the odd cohomology of $\Mb_{g_1,n_1+1}$ (respectively, $\Mb_{g_2,n_2+1}$) vanishes. Therefore, for $3\leq g\leq 8$ and $n<d(g)$, the terms \[\H^{p}(\Mb_{g_1,n_1+1})\otimes \H^{q}(\Mb_{g_2,n_2+1})\] with $p$ and $q$ odd vanish identically, as either $\H^{p}(\Mb_{g_1,n_1+1})$ or $\H^{q}(\Mb_{g_2,n_2+1})$ is zero. \qed 
\bibliographystyle{amsplain}
\bibliography{refs}
\end{document}